\documentclass[11pt,a4paper,reqno]{amsart}
\usepackage{amsmath,amssymb,amsfonts,epsfig,mathrsfs,cite, hyperref}
\usepackage[T1]{fontenc}
\usepackage{color}
\usepackage{array}
\usepackage{amsthm}
\usepackage{amstext}
\usepackage{graphicx}
\usepackage{setspace}

\usepackage[title]{appendix}

\makeatletter
\@namedef{subjclassname@2020}{%
  \textup{2020} Mathematics Subject Classification}
\makeatother

\usepackage[margin=2.5cm]{geometry}
\usepackage{color}
\usepackage{enumitem}
\usepackage{amscd,psfrag}
\usepackage{yhmath}

\usepackage{comment}

\allowdisplaybreaks[4]

\usepackage{slashed}

\makeatletter
\usepackage{epstopdf}

\usepackage{indentfirst}	

\usepackage[normalem]{ulem}
\theoremstyle{plain}

\newtheorem{definition}{Definition}
\newtheorem{theorem}[definition]{Theorem}
\newtheorem*{theorem*}{Theorem}

\newtheorem{remark}[definition]{Remark}

\newtheorem*{remark*}{Remark}
\newtheorem*{sideremark*}{Side Remark}

\newtheorem*{claim*}{Claim}
\newtheorem*{q*}{Question}
\newtheorem{lemma}[definition]{Lemma}

\newtheorem*{corollary*}{Corollary}

\newtheorem{proposition}[definition]{Proposition}

\newcommand{\R}{\mathbb{R}}

\newcommand{\na}{\nabla}

\newcommand{\p}{\partial}

\newcommand{\e}{\varepsilon}

\newcommand{\dd}{{\rm d}}

\newcommand{\zero}{{\underline{\bf 0}}}

\newcommand{\sm}{{\mathscr{S}_M}}

\newcommand{\B}{{\bf B}}

\newcommand{\baru}{{\overline{U}}}

\def\XXint#1#2#3{{\setbox0=\hbox{$#1{#2#3}{\int}$ }
\vcenter{\hbox{$#2#3$ }}\kern-.6\wd0}}

\newcommand{\co}{{\overline{\Omega}}}
\newcommand{\po}{{\partial{\Omega}}}

\newcommand{\tu}{{\widetilde{U}}}
\newcommand{\haus}{{\mathcal{H}^2}}

\newcommand{\stokes}{{\mathbf{A}}}

\newcommand{\bilin}{{\mathscr{B}}}

\newcommand{\n}{{\bf n}}

\title{Boundary epsilon regularity for incompressible Navier--Stokes equations via weak-strong uniqueness}

\author{Siran Li}

\address{Siran Li: School of Mathematical Sciences $\&$ CMA-Shanghai, Shanghai Jiao Tong University, No.~6 Science Buildings,
800 Dongchuan Road, Minhang District, Shanghai, China (200240)}

\email{\texttt{siran.li@sjtu.edu.cn}}


\date{\today}

\begin{document}

\begin{abstract}
We show that finite-energy weak solutions to the incompressible Navier--Stokes equations on a three-dimensional bounded smooth domain are regular up to the boundary, provided that the $L^4_tL^4_x$-norm of the solution is smaller than a constant depending only on the domain. This answers a problem raised in [D. Albritton, T. Barker, and C. Prange, \textit{J. Math. Fluid Mech.} \textbf{25} (2023), Paper No. 49]. Our proof relies on a new slicing construction near the boundary of the domain.

\end{abstract}
\maketitle

\section{Introduction}\label{sec: intro}

This paper is concerned with the regularity theory \emph{up to the boundary} of the incompressible Navier--Stokes equations on a smooth bounded domain $\Omega \subset \R^3$:
\begin{equation}\label{NSE}
\begin{cases}
\p_t U + U \cdot \na U - \Delta U + \na P = 0,\\
\na \cdot U = 0\qquad \text{in } ]-1,0[\times\Omega.
\end{cases}
\end{equation}
Here, $U: ]t_1,t_2[\times\Omega\to\R^3$ is the velocity and $p: ]t_1,t_2[ \times \Omega \to \R$ is the pressure of an incompressible fluid flow for some $-\infty \leq t_1<t_2 \leq \infty$. The kinematic viscosity is normalised to $1$.

The partial regularity theory for the incompressible Navier--Stokes equations has long been a central topic in mathematical hydrodynamics and the analysis of PDE \cite{leray, Lady, fef, lions, tao, hopf}. In this direction, the epsilon regularity theorem at one scale has played a foundational role in the proof of some of the most important results in the  regularity theory for Navier--Stokes; \emph{cf.} Caffarelli--Kohn--Nirenberg \cite{CKN}, Escauriaza--Seregin--\u{S}ver\'{a}k \cite{ESS},  Ne\u{c}as--R${\rm \mathring{u}}$\v{z}i\v{c}ka--\u{S}ver\'{a}k \cite{NRS}, and Tsai \cite{Tsai}, etc. Various proofs of the epsilon regularity theorem can be found in the literature, either via direct iteration \cite{CKN, LR} and De Giorgi--Nash--Moser iteration arguments \cite{Vas, Wang}, or via compactness arguments \cite{Lin, LS}.

In the recent nice work \cite{ABP}, Albritton--Barker--Prange (2023) gave a  new proof of the one-scale epsilon regularity theorem, based on a weak-strong uniqueness argument and the theory of \emph{very weak solutions} to the Navier--Stokes equations pioneered by Foias \cite{foias}, Amann \cite{ama1, ama2}, and Farwig--Galdi--Sohr \cite{FGS}.\footnote{The authors of \cite{ABP} attributed ``the genesis of the paper'' to ``an offhanded comment in 2017'' by Vladim\'{i}r \u{S}ver\'{a}k.} The following interior epsilon regularity theorem --- the quantitative part of \cite[Theorem~B]{ABP} --- ascertains that if the first-time singularities of a finite-energy weak solution $U$ to the Navier--Stokes equations do not occur strictly before time $t=0$, and if the spacetime $L^4$-norm of $U$ over $]-1,0[\times\B_1$ is smaller than a universal constant $\overline{\e}>0$, then $U$ is essentially bounded and hence smooth (\emph{cf}. Serrin \cite{Ser}) up to time $0$ in the interior of $\Omega$.  

\begin{theorem}\label{thm: ABP}
There exists $\overline{\e} \in ]0, 1[$ such that for any $M \in ]0,\infty[$, for any finite-energy weak solution $U$ to the Navier--Stokes equations belonging to $C^\infty\left(]-1,T[ \times \B_1\right)$ for all $T \in ]-1,0[$ such that $\|U\|_{L^\infty_t L^2_x\left(]-1,0[\times\B_1\right)} + \|\na U\|_{L^2_tL^2_x\left(]-1,0[\times\B_1\right)} \leq M,$ the following result holds. 

Let $0<\e<\overline{\e}$. Assume that $\|U\|_{L^4\left(]-1,0[ \times \B_1\right)} \leq \e$. Then $U$ is essentially bounded in $]-1,0[ \times \B_{1/4}$ with $\|U\|_{L^\infty\left(]-1,0[ \times \B_{1/4}\right)} \leq P(M,\e)$, where $P$ is a polynomial. 
\end{theorem}

For convenience of reference, we denote\footnote{For a function $f=f(x)$, we write $\|f\|_{\dot{W}^{1,2}_x(\Omega)} := \left\{\int_{\Omega} |\na f(x)|^2\,\dd x\right\}^{1/2} \equiv \|\na f\|_{L^2_x(\Omega)}$. }
\begin{align}
\label{finite energy weak sol, def}    \sm(]t_1,t_2[\times\Omega]) &:= \bigg\{U:\,\text{ $U$ is a finite-energy weak solution in $]t_1,t_2[\times\Omega$ } \nonumber \\
    &\qquad \qquad \text{ with }\|U\|_{\left(L^\infty_t L^2_x \cap L^2_t \dot{W}^{1,2}_x\right)\left(]t_1,t_2[\times\Omega\right)} \leq M \bigg\}.
\end{align}

\begin{definition}
A function $U: ]t_1,t_2[\times\Omega \to \R^3$  is said to be a finite-energy weak solution to the Navier--Stokes equations in $]t_1,t_2[\times\Omega$ if $U$ satisfies
\begin{align*}
    \p_t U + U \cdot \na U - \Delta U + \na P = 0,\qquad \na \cdot U = 0\qquad \text{ in } \mathcal{D'}\left(]t_1,t_2[\times\Omega\right)
\end{align*}
and the norm $\|U\|_{\left(L^\infty_t L^2_x \cap L^2_t \dot{W}^{1,2}_x\right)\left(]t_1,t_2[\times\Omega\right)}$ is finite. 
\end{definition}

A brief outline of the proof in \cite{ABP} is given below. It is a perturbative argument: thanks to the hypothesis $\|U\|_{L^4_t L^4_x} \leq \e < \overline{\e}$, one may regard the very weak solution to the Navier--Stokes equations as a perturbation of the very weak solution to the Stokes problem. Farwig--Galdi--Sohr \cite{FGS} developed a linear theory for initial-boundary value problems for the very weak solutions. The major technicality for its application to the epsilon regularity lies in the selection of appropriate initial-boundary data.

Consider the initial-boundary value problem for the  Stokes problem on a smooth\footnote{By \cite{FGS, FKS}, $C^{2,1}$-regularity for the domain $\Omega$ suffices here.} bounded domain $\Omega \subset \R^3$ and $-\infty< t_1<t_2<\infty$:
\begin{equation}\label{stokes}
\begin{cases}
    &\p_t U -\Delta U + \na P = \na \cdot F,\qquad \na \cdot U=0\qquad \text{ in }]t_1,t_2[\times\Omega,\\
    &U = a \qquad \text{ on } ]t_1,t_2[ \times \po,\\
    &U=b\qquad \text{ on } \{t_1\} \times \Omega.
\end{cases}
\end{equation}
In the case $F\equiv 0$, $b=0$,\footnote{Subject to some technical compatibility conditions for the initial and boundary values. See \S\ref{sec: prelim} for details.} this linear problem has a unique very weak solution in the regularity class $L^4_t L^6_x$, such that $\|U\|_{L^4_t L^6_x} \lesssim \|a\|_{L^4_t L^4_x}$; see \cite{FGS, FKS}. Thus, under the smallness assumption $\|U\|_{L^4_t L^4_x} \leq \e < \overline{\e}$, by a fixed point argument applied to the mild solution (\emph{i.e.}, weak solution represented as an integral of the forcing term and initial-boundary data via the Stokes semigroup) to the Stokes problem, we may obtain a unique weak solution in $L^4_t L^6_x$ for the Navier--Stokes equations, \emph{i.e.}, Equation~\eqref{stokes} with the forcing term $F = U \otimes U$. Then, by a weak strong uniqueness argument, this unique weak solution coincides with $U$. The contractivity of the solution operator needed to run the fixed point argument is guaranteed by the smallness of the universal constant $\overline{\e}$. Moreover, the region where we apply the above arguments --- which is required to have $L^4_x$-small initial data and $L^4_tL^4_x$-small boundary data --- can be chosen in view of the $L^4_tL^4_x$-smallness of $U$ in $\Omega$ via pigeonholing over the spheres centred at $(t,x)=(0,\zero)$. This process is referred to as spatial-temporal slicing in \cite{ABP}. Once we promote the regularity of $U$ from $L^4_t L^4_x$ to $L^4_t L^6_x$, the latter falls within the range of indices entailed by the Lady\v{z}enskaja--Prodi--Serrin regularity criterion \cite{Lady, prodi, Ser}, which then bootstraps the interior regularity of $U$ to $L^\infty_t L^\infty_x$.

The following question is posed in \cite{ABP}: {\bf Can we establish an epsilon regularity theorem \emph{up to the boundary} for the Navier--Stokes equations based on the weak strong uniqueness theory for very weak solutions?} Let us quote \cite[Remark~4.2]{ABP}:

\begin{quote}
It remains an open problem as to whether such a proof can be done for
establishing epsilon regularity results near a boundary. Indeed, the linear results of Farwig \textit{et al.} \cite{FKS} and
of Fabes \textit{et al.} \cite{FLR2} ask for smoothness of the domain $\Omega$, see for instance Theorem 2.2 above. However, we
are unable to carry out the slicing procedure of Step~1 [in the proof of Theorem~B] above near a smooth boundary.
\end{quote}

To address this problem, two difficulties must inevitably be encountered. They both echo the above quoted remark in that an effective slicing or foliation of the fluid domain arbitrarily close to the boundary remains elusive.
\begin{enumerate}
    \item 
    First, assuming that $U$ has small $L^4_tL^4_x$-norm in $]-1,0[\times\Omega$ does not imply that the trace of $U$ on $\p\Omega$ has  small $L^4_tL^4_x$-norm in $]-1,0[\times \p\Omega$. Indeed, for a finite-energy weak solution $U \in \left(L^\infty_t L^2_x \cap L^2_t \dot{W}^{1,2}_x\right)\left(]-1,0[\times\Omega\right)$, one may only infer from the trace inequality that $U \in L^2_tL^4_x\left(]-1,0[\times \p\Omega\right)$, which is a weaker result than $U \in L^4_tL^4_x\left(]-1,0[\times \p\Omega\right)$.\footnote{Throughout, the notation $u \in L^p_tL^q_x(]t_1,t_2[\times \po)$ and the like are  understood in the sense of trace.} 
    \item 
    Second, to perform the aforementioned slicing argument near the boundary, merely slicing over concentric spheres in space is no longer sufficient, as it only leads to interior regularity results. On the other hand, fixing a putative singular point on the boundary, we cannot simply slice by hemispheres centred at this point, because by Item~(1) above we do not have $L^4_tL^4_x$-smallness of the trace of $U$ on $\po$.
\end{enumerate}

The goal of this paper is to prove the  boundary epsilon regularity theorem as follows. Other boundary regularity results for the Navier--Stokes equations have been obtained in Seregin \cite{S03}, Seregin--Shilkin--Solonnikov \cite{sss}, Albritton--Barker \cite{ab}, and Breit \cite{breit}, etc.

\begin{theorem}\label{thm: main}
Let $\Omega \subset \R^3$ be a smooth bounded domain. There exists a constant $\e_0>0$ such that the following holds. Let $U \in \sm\left(]-1,0[\times\Omega\right)$ be a finite-energy weak solution to the Navier--Stokes equations. Assume that for any $T\in ]-1,0[$ one has $U \in C^\infty\left(]-1,T[ \times {\Omega}\right)$, and that  $\left\|U\right\|_{L^4\left(]-1,0[\times\Omega\right)}<\e_0.$ Then $U \in C^\infty\left(]-1,0]\times\co\right)$. In addition, $\|U\|_{C^\infty\left(]-1,0]\times\co\right)} \leq C'\left\|U\right\|_{L^4\left(]-1,0[\times\Omega\right)}$ for a constant $C'$ depending only on $M$ and $\Omega$. 
 
\end{theorem}
 
The key tool for the proof of  Theorem~\ref{thm: main} is a new spatial slicing construction in Proposition~\ref{prop, clam}, which may be of independent interest.  Assume $\po = \p\R^3_+$ and fix any $x_\star \in \po$. One may find a paraboloid tangential to the boundary only at $x_\star$. Take a smooth convex region $\mathscr{V}_0$ enclosed by the paraboloid from  below and some spherical cap from above. If we contract it to $\{x_\star\}$ along the geodesics connecting $x_\star$ and the points on $\p\mathscr{V}_0$, we obtain a foliation of $\mathscr{V}_0$ such that each folium $\Sigma_s$ is a smooth convex region, whose bottom part is a paraboloid that touches $\p\Omega$ at a single point $x_\star$ in a $C^1$-tangential manner. In addition, it is clear that the annulus between the contracted regions at time $s=1/4$ and $s=1/2$ has large volume, due to the convexity of $\mathscr{V}_0$ and the regions obtained along the contraction. Hence, by pigeonholing over this foliation, we may find some $s \in [0,1[$ such that the trace of $U$ over $\Sigma_s$ has a small $L^4_tL^4_x$-norm. 

The region enclosed by $\Sigma_s$ has zero distance from the boundary $\po$, and all the estimates involved in the preceding process are uniform in the base point $x_\star \in \po$. The above construction thus allows us to deduce epsilon regularity results up to the boundary, namely Theorem~\ref{thm: main}.

\smallskip
\noindent
{\bf Organisation.} The remaining parts of the paper are organised as follows. In \S\ref{sec: prelim} we collect some background knowledge on the theory of very weak solutions and the Lady\v{z}enskaja--Prodi--Serrin regularity criterion. In \S\ref{sec: slicing} we present the crucial novel slicing technique in this paper. The Main Theorem~\ref{thm: main} is proved in \S\ref{sec: proof}. Finally, a few concluding remarks are given in \S\ref{sec: concluding remarks}.

\section{Preliminaries}\label{sec: prelim}

The notion of \emph{very weak solutions} to the Stokes equations is central to our development. The following formulation is taken from Amann \cite{ama1, ama2}, Farwig--Galdi--Sohr \cite{FGS}, and  Farwig--Kozono--Sohr \cite{FKS}, whereof the case $F = U\otimes U$ (\emph{e.g.}, for $U \in L^2_tL^2_x\left(]t_1,t_2[\times\Omega\right)$) corresponds to the Navier--Stokes equations. 

\begin{definition}\label{def: very weak sol}
Let $\Omega$ be a domain in $\R^3$ and $-\infty \leq t_1 < t_2 \leq +\infty$. Given $a \in L^1\left(]t_1,t_2[\times\po;\R^3\right)$, $b \in L^1\left(\{t_1\} \times \Omega;\R^3\right)$, and $F \in L^1\left(]t_1,t_2[\times\Omega; \R^{3\times 3}\right)$ satisfying the compatibility conditions 
\begin{align*}
    \int_{\po} a(t,x)\cdot\n(x)\,\dd\haus(x) = 0 \quad\text{and}\quad \int_\Omega b(x)\cdot\na \psi(x)\,\dd x = 0\quad\text{for all } t \in ]t_1,t_2[\text{ and } \psi \in C^\infty(\R^3).
\end{align*}
An integrable function $U: ]t_1,t_2[\times\Omega \to \R^3$ is said to be a very weak solution to the initial-boundary value problem of the forced Stokes problem:
\begin{equation*}
\begin{cases}
    &\p_t U -\Delta U + \na P = \na \cdot F,\qquad \na \cdot U=0\qquad \text{ in }]t_1,t_2[\times\Omega,\\
    &U = a \qquad \text{ on } ]t_1,t_2[ \times \po,\\
    &U=b\qquad \text{ on } \{t_1\} \times \Omega,
\end{cases}
\end{equation*}
if for all test functions $\Phi \in C^1_0\left([t_1,t_2[; C^2_{0,\sigma}\left(\co\right)\right)$ and $q \in C^\infty\left([t_1,t_2]\times\co;\R\right)$, it holds that
\begin{align*}
&-\int_{t_1}^{t_2}\int_\Omega U \cdot \left\{ \p_t\Phi + \Delta\Phi +\na q \right\}\,\dd x\,\dd t =  -\int_{t_1}^{t_2}\int_\Omega F:\na\Phi\,\dd x\,\dd t\\
 &\qquad + \int_\Omega b(x)\cdot\Phi(t_1,x)\,\dd x -  \int_{t_1}^{t_2}\int_\po \left\{a\cdot\left(\na\Phi\cdot\n\right)  + \left(a \cdot \n\right) q\right\}\,\dd\haus(x)\,\dd t.
\end{align*}
Here and hereafter, $C^{2}_{0,\sigma}\left(\co\right):=\left\{u \in C^2\left(\co;\R^3\right):\, u\big|_\po \equiv 0,\,\na\cdot u  = 0 \right\}$. 
\end{definition}

The following existence and uniqueness theorem has been established in \cite[Lemma~1.2]{FKS}; see also \cite[Theorem~2.2]{ABP}. In this paper, we shall take $s=4$ and $q=6$. 
\begin{proposition}\label{prop: very weak}
    Let $s,q \in [4,\infty[$ with $\frac{2}{s} + \frac{3}{q} = 1$. Let $a \in L^s_tL^{\frac{2q}{3}}_x(]t_1,t_2[ \times \Omega)$ and $F\equiv 0$, $b\equiv 0$. Assume that $a$ satisfies the compatibility condition in Definition~\ref{def: very weak sol}. Then there exists a unique very weak solution $U \in  L^s_tL^q_x(]t_1,t_2[ \times \Omega)$ to the Stokes problem~\eqref{stokes}. In addition,
    \begin{align}
        \|U\|_{L^s_tL^q_x(]t_1,t_2[ \times \Omega)} \leq C\|a\|_{L^s_tL^{\frac{2q}{3}}_x(]t_1,t_2[ \times \Omega)}
    \end{align}
    for some constant $C$ depending only on $q$, $t_1$, $t_2$, and $\Omega$. 
\end{proposition}

We also have the following result on the uniqueness of $L^2_tL^2_x$-very weak solutions and the energy identity. See \cite[lemma~2.5 and Remark~2.6]{ABP} and \cite[IV, Theorems~2.3.1 and 2.4.1]{sohr}.
\begin{proposition}\label{prop: uniqueness and energy}
In the setting of Definition~\ref{def: very weak sol}, let $U \in L^2_tL^2_x\left(]t_1,t_2[\times\Omega\right)$ be a very weak solution to the linear Stokes problem~\eqref{stokes}. 
\begin{enumerate}
    \item 
    If $F$, $a$, $b$ are all constantly zero, then $U \equiv 0$ on $]t_1,t_2[\times\Omega$.
    \item 
    If $a$ and $b$ are constantly zero, and if $F \in \left(]t_1,t_2[\times\Omega; \R^{3\times 3}\right)$, then $U \in C^0_tL^2_{\sigma,x}\left([t_1,t_2] \times \Omega\right) \cap L^2_t W^{1,2}_{0,\sigma,x}\left(]t_1,t_2[\times\Omega\right)$ with the energy identity:
    \begin{equation*}
    \|U(t,\bullet)\|^2_{L^2_x(\Omega)} + 2\int_{t_1}^t\int_\Omega |\na U(x,s)|^2\,\dd x \,\dd s = -2\int_{t_1}^t \int_\Omega F: \na U\,\dd x\,\dd s\qquad \text{for all } t \in [t_1,t_2].
    \end{equation*}
\end{enumerate}
\end{proposition}

Finally, recall the following celebrated result by Lady\v{z}enskaja--Prodi--Serrin \cite{Lady, prodi, Ser}.\footnote{Throughout this paper, $\B_r(x_0)$ is the Euclidean ball in $\R^3$ centred at $x_0$ with radius $r$, and $\B_r^+(x_0) := \B_r(x_0) \cap \R^3_+$. The boldface $\zero$ always denotes the origin in $\R^3$.}
\begin{theorem}\label{thm: LPS}
    Let $U \in \mathscr{S}_M\left(]-1,0[ \times \B_{r_0}(x_0)\right)$ be a finite-energy weak solution to the Navier--Stokes equations. Suppose for some $s \in [2,\infty[$ and $r \in ]3,\infty]$ that
    \begin{align*}
        U \in L^s_tL^r_x\left(]-1,0[\times\B_{r_0}(x_0)\right)\qquad\text{with } \frac{2}{s}+\frac{3}{r} \leq 1.
    \end{align*}
    Then $U$ is smooth with respect to the space variables in $\B_{r_0/2}(x_0)$. The boundary regularity version of this result also holds, \emph{i.e.}, with $\B_{r_0}^+(x_0)$ and $\B_{r_0/2}^+(x_0)$ in place of $\B_{r_0}(x_0)$ and $\B_{r_0/2}(x_0)$, respectively.
\end{theorem}

The proof of Theorem~\ref{thm: LPS} follows essentially from energy considerations. Note that $L^s_tL^r_x$ with $2/s+3/r=1$ are scale-invariant spaces for the solutions to the Navier--Stokes equations. The  endpoint case $(s,r)=(\infty,3)$ was established in Escauriaza--Seregin--\u{S}ver\'{a}k \cite{ESS}, which relies on delicate applications of the Carleman estimates for the backward uniqueness of the heat equation; see also Tao \cite{tao} for  quantitative refinements.

\section{Slicing by clams}\label{sec: slicing}

We construct the following foliation of a convex body compactly supported in the upper halfspace $\overline{\R^3_+}$.

\begin{proposition}\label{prop, clam}
There exists $\mathscr{V}_0 \subset \left[-\frac{1}{2},\frac{1}{2}\right] \times \left[-\frac{1}{2}, \frac{1}{2}\right]  \times [0,1]$ such that the following holds.
\begin{enumerate}
    \item 
     $\mathscr{V}_0$ is a smooth convex body rotationally symmetric with respect to the $z$-axis.
     \item 
     $\overline{\mathscr{V}_0} = \bigcup_{s \in [0,1[} \Sigma_s$ where each $\Sigma_s$ is a smooth convex surface. Also $\Sigma_s \cap \Sigma_{s'} = \{\zero\}$ and $T_\zero\Sigma_s = T_\zero \Sigma_{s'} = \p\R^3_+$ for any $s \neq s'$ in $[0,1[$.

     \item 
     There is a smooth function $z_0: \left[0,1\right] \to \left[0,\frac{1}{4}\right]$ with $z_0(s) \searrow 0$ as $s \nearrow 1$, such that $\Sigma_s\cap \left\{0 \leq z \leq z_0(s)\right\}$ is a paraboloid and $\Sigma_s \cap \left\{4z_0(s) \leq z \leq 1\right\}$ is a spherical cap for each $s \in [0,1[$. 

     \item 
     There exists a universal constant $c_0>0$ such that the  volume of $\bigcup_{s \in \left] \frac{1}{4}, \frac{1}{2} \right[} \Sigma_s \geq c_0$. 
     
\end{enumerate}

\end{proposition}

In other words, we construct a smooth convex body $\mathscr{V}_0 \subset \R^3$ foliated by convex surfaces $\Sigma_s$ that intersect $C^1$-tangentially at the origin $\zero$. For each folium $\Sigma_s$, its upper part is a spherical cap and its lower part is a paraboloid. The shape of $\mathscr{V}_0$ is like a clam.

Several notations are used in this section. For $a, b \in \R^3$, we write $[a,b]$ for the straight line segment in $\R^3$ with endpoints $a$, $b$; similarly for $]a,b[$, $[a,b[$, and $]a,b]$. Also, for $S \subset \R^3$ and $\lambda \geq 0$ we write $\lambda S:=\{\lambda x:\, x \in S\} \subset \R^3$.

\begin{proof}[Proof of Proposition~\ref{prop, clam}]
We first construct $\Sigma_0 = \p\mathscr{V}_0$ as follows. Consider the paraboloid $$\mathscr{P}_0 = \left\{(y',y^3) \in \R^2\times \R:\, y^3 = \frac{|y'|^2}{4},\, |y'| \leq 1 \right\}.$$ Then at $|y'|=\frac{1}{\sqrt{2}}$ one has $y^3=\frac{1}{8}$ for $(y',y^3) \in \mathscr{P}_0$, while $y^3 = \frac{1}{2}$ for $(y',y^3) \in \p\B_1^+(\zero)$, the hemisphere of radius $1$. An easy gluing of $\mathscr{P}_0 \cap \{y^3 \leq \frac{1}{8}\}$ and $\p\B_1^+(\zero) \cap \{y^3 \geq \frac{1}{2}\}$ utilising the standard mollifier yields a smooth, (strictly) convex surface $\Sigma_0 \subset \left[-\frac{1}{2},\frac{1}{2}\right] \times \left[-\frac{1}{2}, \frac{1}{2}\right]  \times [0,1]$ that is rotationally symmetric with respect to the $z$-axis. Note that $\Sigma_0$ intersects $\p\R^3_+$ at the origin with $T_\zero\Sigma_0 = \p\R^3_+$.

Denote by $\mathscr{V}_0$ the convex body enclosed by $\Sigma_0$. It verifies Condition~(1) in the proposition.

Note that $\mathscr{V}_0$ is star-shaped with respect to $\zero$, and $$\overline{\mathscr{V}_0} = \bigcup_{x \in \Sigma_0} [0,x].$$ Then, observe that
\begin{align*}
    {\mathscr{V}_s} := \bigcup_{x \in \Sigma_0} ]0,(1-s)x[ \equiv (1-s)\mathscr{V}_0
\end{align*}
is a convex body for each $s \in ]0,1[$, while $\Sigma_s := \p\mathscr{V}_s$ satisfies $\overline{\mathscr{V}_0} = \bigcup_{s \in [0,1[} \Sigma_s$ and $\Sigma_s \cap \Sigma_{s'} = \{\zero\}$ for distinct $s$ and $s'$. We also observe that, for each $s \in [0,1[$, the part of $\Sigma_s$ in the vicinity of the origin $\zero$ is a paraboloid. Indeed, it takes the form $(1-s)\mathscr{P}_0$, and 
\begin{align*}
    x^3 = \frac{|x'|^2}{4(1-s)}\qquad \text{for each } x=(x',x^3) \in (1-s)\mathscr{P}_0.
\end{align*}
That is, $(1-s)\mathscr{P}_0$ is a paraboloid with a larger curvature than $\mathscr{P}_0$, which intersects the halfspace $\p\R^3_+$  in the $C^1$-tangential manner at $\zero$. This verifies Condition~(2) in the proposition. Earlier arguments in this paragraph and the choice $$z_0(s)=\frac{1-s}{8}$$ together verify Condition~(3).

Finally, denote by $J(K)$  the John ellipsoid of convex body $K$. Then
\begin{align*}
    {\rm Volume}\left(\bigcup_{s \in \left] \frac{1}{4}, \frac{1}{2} \right[} \Sigma_s \right) & =  {\rm Volume}\left(\mathscr{V}_{\frac{1}{4}} \setminus \mathscr{V}_{\frac{1}{2}} \right)\\
    &\geq  \frac{1}{4}{\rm Volume}\left(\mathscr{V}_{\frac{1}{4}} \right) \\
    &\geq \frac{1}{4}{\rm Volume}\left(J\left(\mathscr{V}_{\frac{1}{4}} \right)\right) =: c_0. 
\end{align*}
This verifies Condition~(4).  
\end{proof}

\begin{remark}
By adapting the proof above, one may construct a smooth convex body $\mathscr{V}_0$ such that the graphing function of each folium $\Sigma_s$ has arbitrary vanishing order at the origin for each $s \in [0,1[$; \emph{i.e.}, for each fixed $\ell \in \mathbb{N}$, $\Sigma_s$ intersects the halfspace $\p\R^3_+$ in the $C^\ell$-tangential manner at $\zero$. To this end, we simply replace $\mathscr{P}_0$ by $$\left\{(y',y^3) \in \R^2\times \R:\, y^3 = c |y'|^{\ell+1},\, |y'| \leq 1 \right\}$$ with some suitable $c=c(\ell)$. 

\end{remark}

\section{The proof}\label{sec: proof}

This section is devoted to the proof of the Main Theorem~\ref{thm: main}.

\begin{proof}
The interior epsilon regularity theorem has already been established in \cite{ABP}, so it suffices to establish boundary regularity only. All the constants $C_i$, $i \in \mathscr{N}$ in this proof are constants depending only on the $C^2$-geometry of the domain $\Omega$.

We divide the arguments into eight steps below. 

\smallskip
\noindent
{\bf Step~1: Boundary straightening.} Since $\Omega \Subset \R^3$ is a bounded smooth domain, there exists a positive constant $r_0$ depending only the $C^2$-geometry of $\Omega$ such that the following holds: for each $x_\star \in \po$, there exists a diffeomorphism $\Phi_{x_\star}$ defined on $\B_{32r_0}(x_\star)\cap\Omega$ which flattens the boundary near $x_\star$ and has bilipschitz constant no larger than $2$ (namely that $\Phi_{x_\star}$ is 2-bilipschitz). It is crucial that the choice of $r_0$ is uniform in $x_\star \in \po$.

Recall that $\B_{r}(x)$ denotes the Euclidean ball of centre $x$ and radius $r$ in $\R^3$ and $\B^+_{r}(x) := \B_{r}(x) \cap \R^3_+$ is the upper half ball; $\B_r \equiv \B_r(0)$ and $\B_r^+ \equiv \B_r^+(0)$. Now we fix, once and for all thoughout this proof, a point $x_\star \in \po$ and a number $$r_0 \in \left]0,\, \frac{{\rm inj\,rad}\,(\po)}{64}\right[.$$
Here and hereafter, ${\rm inj\,rad}\,(\po)$ is the injectivity radius of the boundary $\po$, which depends only on the $C^2$-geometry of $\Omega$. That is, the collar $\left\{x \in \Omega:\, {\rm dist}(x,\po)<{\rm inj\,rad}\,(\po)\right\}$ is globally diffeomorphic to the product $\po \times ]0,{\rm inj\,rad}\,(\po)[$ via the exponential map starting from the points on $\po$ along the inward normal vector field.

Let us designate
\begin{equation}\label{tilde-U, def}
    \tu (t,y) := U \left(t, \Phi^{-1}_{x_\star}(y)\right) \in L^4_t L^4_x \left(]-1,0[ \times\B^+_{64r_0} \right).
\end{equation}
Then 
\begin{align*}
    \left\| \tu\right\|_{L^4_t L^4_x \left(]-1,0[ \times\B^+_{32r_0} \right)} < 2\e_0.
\end{align*}

\smallskip
\noindent
{\bf Step~2: Spatial slicing.} Recall the clam-like smooth convex region $\mathscr{V}_0$ constructed in Proposition~\ref{prop, clam}. Consider $$\mathscr{V}' := 16r_0\mathscr{V}_0$$ together with its foliation
\begin{align*}
\overline{\mathscr{V}'} = \bigcup_{s \in [0,1[} \Sigma'_s
\end{align*}
with 
$$\Sigma_s' := 16r_0\Sigma_s.$$ The symbols $\mathscr{V}_0$, $\Sigma_s$ are as in Proposition~\ref{prop, clam}.

By Fubini's theorem, we have
\begin{align*}
\int_{-1}^0 \int_{4r_0}^{8r_0} \int_{\Sigma'_s} \left|\tu\right|^4\,\dd\haus\,\dd s \,\dd t \leq \left\|\tu \right\|_{L^4_tL^4_x\left(]-1,0[ \times \B^+_{32r_0}\right)}^4 < (2\e_0)^4.
\end{align*}
We then infer from the pigeon hole principle that
\begin{align}\label{spatial slice}
\int_{-1}^0 \int_{\Sigma'_s} \left|\tu\right|^4\,\dd\haus\,\dd t \leq 8\left\| \tu\right\|^4_{L^4_t L^4_x \left(]-1,0[ \times\B^+_{32r_0} \right)} < 128\,\e_0^4 
\end{align}
for some $s \in ]4r_0,8r_0[$.

For future references, from now on we designate
\begin{equation}\label{region V}
\mathscr{V} \equiv \mathscr{V}_{s,x_\star}  := \text{ the smooth, convex region whose boundary is $\Sigma'_s$}.
\end{equation}

\smallskip
\noindent
{\bf Step~3: Temporal slicing.} Again, by pigeon hole and Fubini, we have
\begin{equation}\label{temporal slice}
    \int_{\B^+_{32 r_0}} \left|\tu(t_0,x)\right|^4\,\dd x  \leq 8 \left\| \tu\right\|^4_{L^4_t L^4_x \left(]-1,0[ \times\B^+_{32r_0} \right)} <128 \e_0^4
\end{equation}
for some $t_0 \in \left]-1,-\frac{7}{8}\right[$.  

\smallskip
\noindent
{\bf Step~4: Boundary and initial data.}  Recall the nice region $\mathscr{V}$ from Equation~\eqref{region V}. Set 
\begin{equation*}
    \widetilde{a}:= \tu\Big|_{]-1,0[ \times \p\mathscr{V}},\qquad \widetilde{b}:= \tu\Big|_{\{t=t_0\}}.
\end{equation*}
In view of Equations~\eqref{spatial slice} and \eqref{temporal slice}, one has
\begin{align}\label{a, b-tilde}
\left\|\widetilde{a}\right\|_{L^4_tL^4_x\left({]-1,0[ \times \p\mathscr{V}}\right)} + \left\|\widetilde{b}\right\|_{L^4_x\left(\{t_0\} \times \B^+_{32r_0} \right)} &\leq 2^{\frac{7}{4} } \left\| \tu\right\|_{L^4_t L^4_x \left(]-1,0[ \times\B^+_{32 r_0} \right)} \nonumber\\
    &< 2^{\frac{11}{4}}\e_0.
\end{align}
Also recall the diffeomorphism $\Phi_{x_\star}$ defined in Step~1 of the proof. Set \begin{equation}\label{R, def}
    \mathscr{R} := \Phi_{x_\star}^{-1} \left(\mathscr{V}\right) \subset \Omega.
\end{equation}
Note that $\overline{\mathscr{R}}$ intersects $\po$ in the singleton $x_\star$, with $$T_{x_\star}\left(\p{\mathscr{R}}\right) = T_{x_\star}\left(\po\right).$$

We have the following:
\begin{align}\label{a,b-def}
    a := U\Big|_{]-1,0[ \times \p \mathscr{R}},\qquad b:= U\Big|_{\{t=t_0\}}
\end{align}
satisfy that 
\begin{align}\label{a, b estimate}
\left\|{a}\right\|_{L^4_tL^4_x\left({]-1,0[ \times \p\mathscr{R}}\right)} + \left\|{b}\right\|_{L^4_x\left(\{t_0\} \times \Phi_{\star}^{-1}\left(\B^+_{16r_0}\right) \right)}  &\leq 8\|U\|_{L^4_tL^4_x\left(]-1,0[ \times \left[\B_{32r_0}(x_\star)\cap\Omega\right]\right)}\nonumber\\
&< 8\e_0,
\end{align}
thanks to Equation~\eqref{a, b-tilde} and the fact that $\Phi_{x_\star}$ is $2$-bilipschitz.

\smallskip
\noindent
{\bf Step~5: Rectifying the boundary data.} We argue as in \cite[\S 3.2, Proof of Theorem~A']{ABP}. The treatment for nonzero initial data $b$ relies on the Stokes semigroup estimates in the $L^p$-framework, which is a classical topic (see, \emph{e.g.}, Giga~\cite{giga}).

Indeed, the quantity $b$ defined in Equation~\eqref{a,b-def} is divergence-free in the sense of distributions on $\{t_0\} \times \Phi_{\star}^{-1}\left(\B^+_{16 r_0}\right)\subset \{t_0\} \times \left(\B_{32 r_0} \cap \Omega\right)$. Also, in view of the definition of $\mathscr{R}$ in Equation~\eqref{R, def} and that $\Phi_{x_\star}$ is 2-bilipschitz, we have $\mathscr{R} \subset \B_{16 r_0}(x_\star) \cap \Omega$.

Let us choose a spatial cutoff function 
\begin{align}\label{varphi, spatial cutoff}
&\varphi \in C^\infty_c\left(\co\right),\quad 0 \leq \varphi \leq 1,\quad \varphi \equiv 1 \text{ on $\B_{10 r_0}(x_\star)\cap \co$},\nonumber\\
&\qquad \varphi \equiv 0\text{ on $\co\setminus \B_{14 r_0}(x_\star)$}, \quad \text{and } |\na \varphi| \leq \frac{1}{r_0}.
\end{align}
Then $\varphi b$ (identified with its one extension-by-zero) defines an extension of $b\big|_{\mathscr{R}}$ on $\{t_0\} \times \Omega$, which is not divergence-free in general.

By the Bogovski\u{i} estimates~\cite{Bog1, Bog2}, one may find a correction term $v_{\rm Bog} \in L^4_x\left(\{t_0\} \times \Omega\right)$ and a constant $C_0$ depending only on $r_0$ (hence only on the $C^2$-geometry of $\Omega$) such that 
\begin{equation*}
    {\rm div}\left(v_{\rm Bog}\right) = -{\rm div}(\varphi b),\qquad \left\|v_{\rm Bog}\right\|_{L^4_x\left(\{t_0\} \times \Omega\right)} \leq C_0 \|b\|_{L^4\left(\{t_0\} \times \left[\B_{16 r_0}(x_\star)\cap\Omega\right]\right)}.
\end{equation*} 

To proceed, define
\begin{equation}\label{Eb}
    E[b] := \varphi b + v_{\rm Bog} \in L^4_x\left(\{t_0\} \times \Omega\right).
\end{equation}
This is a divergence-free extension of $\varphi b$ (hence of $b\big|_{\mathscr{R}}$) with 
\begin{equation}\label{bound for Eb}
    \big\|E[b]\big\|_{L^4_x\left(\{t_0\} \times \Omega\right)} \leq C_1 \|b\|_{L^4_x\left(\{t_0\} \times \left[\B_{16 r_0}(x_\star)\cap\Omega\right]\right)}
\end{equation}
for some $C_1$ depending only on $r_0$. Then we set 
\begin{equation}\label{heat kernel correction}
    v_{\rm bdry} := \mathcal{H}_\Omega(\bullet-t_0) \star E[b],
\end{equation}
where $\mathcal{H}_\Omega$ is the Dirichlet heat kernel on $\Omega$. Recall that $t_0 \in \left]-1,-\frac{7}{8}\right[$. We then have the estimates (\emph{cf}. Fabes--Lewis--Rivi\`{e}re \cite[Lemma~IV.3.2]{FLR2} and the analogues for $\mathcal{H}_\Omega$ \cite{vdB, GSC}):
\begin{equation}\label{v-bdry estimates}
    \begin{cases}
\|v_{\rm bdry}\|_{L^4\left(]t_0,0[ \times \p\mathscr{R}\right)} \leq C_2 \left\{(-t_0)^{\frac{1}{8}} + (-t_0)^{\frac{1}{4}}\right\}\big\|E[b]\big\|_{L^4_x\left(\{t_0\} \times \Omega\right)} \leq C_3 \|b\|_{L^4_x\left(\{t_0\} \times \Omega\right)},\\
\|v_{\rm bdry}\|_{L^4_t L^6_x\left(]t_0,0[ \times \mathscr{R}\right)} \leq C_4(-t_0)^{\frac{1}{8}} \big\|E[b]\big\|_{L^4_x\left(\{t_0\} \times \Omega\right)} \leq C_5 \|b\|_{L^4_x\left(\{t_0\} \times \Omega\right)}.
    \end{cases}
\end{equation}
Here $C_2,\cdots,C_5$ depend only on the $C^2$-geometry of $\Omega$.

Observe that ${\rm div}\,v_{\rm bdry}=0$ in the distributional sense. Hence, 
\begin{align*}
    \int_{\p\mathscr{R}} v_{\rm bdry}\cdot{\bf n}\,\dd\haus =0.
\end{align*}

\smallskip
\noindent
{\bf Step~6: Linear estimates.} Put
\begin{equation}\label{a'}
    a' := a - v_{\rm bdry}\big|_{\p\mathscr{R}}.
\end{equation} 
By Farwig--Kozono--Sohr \cite[Theorem~2.2]{FKS}, there exists a unique very weak solution 
\begin{equation*}
    U_{a'} \in L^4_t L^6_x \left(]t_0,0[ \times \mathscr{R} \right)
\end{equation*}
to the Stokes problem $\p_t U_{a'} - \Delta U_{a'}+\na P=0$ with zero external forcing, zero initial data at $t=t_0$, and boundary data $a'$ on $\p\mathscr{R}$.

Recall the boundary correction term $ v_{\rm bdry}$ from Equation~\eqref{heat kernel correction}. By Farwig--Galdi--Sohr \cite{FGS} and  
Farwig--Kozono--Sohr \cite[Lemma~1.2]{FKS}, 
\begin{equation}\label{bar-U def}
    \overline{U} := U_{a'} +  v_{\rm bdry}  \in L^4_t L^6_x \left(]t_0,0[ \times \mathscr{R} \right)
\end{equation} 
is the unique very weak solution to the Stokes problem $\p_t  \overline{U} - \Delta \overline{U}+\na P=0$ with zero external forcing, initial data $b$ at $t=t_0$, and boundary data $a$ on $\p\mathscr{R}$.

Moreover, $\overline{U}$ satisfies the estimate:
\begin{equation}\label{bar-U estimate}
    \left\|\overline{U}\right\|_{L^4_t L^6_x \left(]t_0,0[ \times \mathscr{R} \right)} \leq C_6\left\{ \|a\|_{L^4\left(]t_0,0[ \times \p\mathscr{R}\right)} + \|b\|_{L^4_x\left(\{t_0\} \times \left[\B_{32r_0}(x_\star)\cap\Omega\right]\right)} \right\},
\end{equation}
where $C_6$ depends only on $r_0$. It then holds by Equation~\eqref{a, b estimate}, it holds that
\begin{equation}\label{bar-U estimate, epsilon}
    \left\|\overline{U}\right\|_{L^4_t L^6_x \left(]t_0,0[ \times \mathscr{R} \right)} < 8C_6\e_0.
\end{equation}

The remaining parts of the proof are essentially the same as in \cite[Proof of Theorem~A]{ABP}. We present the details for the sake of completeness.

\smallskip
\noindent
{\bf Step~7: Fixed point arguments.}  Having obtained the linear Stokes estimate, namely Equation~\eqref{bar-U estimate}, we are ready to deduce the existence and uniqueness of the very weak solution to the Navier--Stokes equation. 

Consider the perturbed Stokes system:
\begin{equation}\label{perturbed stokes}
    \begin{cases}
    \p_t V - \Delta V + \na P = -{\rm div}\left[\left(\baru + V\right)\otimes \left(\baru + V\right)\right]\qquad \text{in } ]t_0,0[ \times \mathscr{R},\\
{\rm div}\, V = 0 \qquad \text{in } ]t_0,0[ \times \mathscr{R},\\
V=0 \qquad \text{on } ]t_0,0[ \times \p\mathscr{R},\\
V\big|_{t=t_0} = 0 \qquad \text{on } \{t_0\} \times \mathscr{R}.
    \end{cases}
\end{equation}
We construct a fixed point for the mild solution
\begin{equation}\label{mild sol}
    V(t,x) = - \int_{t_0}^0 e^{s\stokes}{\rm div}\,\left[\left(\baru + V\right)\otimes \left(\baru + V\right)\right]\,\dd s \equiv \bilin\left[\baru + V; \baru + V\right],
\end{equation}
where $\{e^{-t\stokes}\}_{t>0}$ is the Stokes semigroup and $\stokes$ the Stokes operator. Also, $\bilin$ is a bilinear operator given by
\begin{align*}
    \bilin[D;E] := -\int_{t_0}^0
e^{s\stokes}{\rm div}\,(D \otimes E)\,\dd s.
\end{align*}
We shall also use the linear operator:
\begin{align*}
    L[D](t,\bullet) := -\int_{t_0}^t e^{s\stokes} {\rm div}\left[ D \otimes \baru + \baru \otimes D \right]\,\dd s. 
\end{align*}

From the well-known estimate:
\begin{align*}
    \left\| e^{s\stokes} {\rm div}(D \otimes E) \right\|_{L^6_x} \leq \frac{C_7}{(-s)^{3/4}}\|D(t,\bullet)\|_{L^6_x}\|E(t,\bullet)\|_{L^6_x},
\end{align*}
together with the Hardy--Littlewood--Sobolev inequality, we deduce that
\begin{align}\label{C8 bound}
\big\|B[D;E]\big\|_{L^4_tL^6_x\left(]t_0,0[ \times \mathscr{R}\right)} \leq C_8 \|D\|_{L^4_tL^6_x\left(]t_0,0[ \times \mathscr{R}\right)} \|E\|_{L^4_tL^6_x\left(]t_0,0[ \times \mathscr{R}\right)}
\end{align}
as well as
\begin{align}\label{C9 bound}
\big\|L[D]\big\|_{L^4_tL^6_x\left(]t_0,0[ \times \mathscr{R}\right)} \leq C_9 \|D\|_{L^4_tL^6_x\left(]t_0,0[ \times \mathscr{R}\right)} \big\|\baru\big\|_{L^4_tL^6_x\left(]t_0,0[ \times \mathscr{R}\right)}.
\end{align}
The constant $C_7 = C_7(\Omega)$. It then holds by Gallagher--Iftimie--Planchon \cite[Lemma~4.1]{GIP} that Equation~\eqref{mild sol} has a unique fixed point $V \in L^4_tL^6_x$, provided that
\begin{equation}\label{smallness condition on bar-U}
\big\|\baru\big\|_{L^4_tL^6_x\left(]t_0,0[ \times \mathscr{R}\right)} < \frac{\min\left\{C_{8}^{-1},C_{9}^{-1}\right\}}{4}.
\end{equation}
The parameters $C_8$, $C_9$ are given by  Equations~\eqref{C8 bound} and \eqref{C9 bound}. In addition, the solution $V$ satisfies
\begin{align*}
    \|V\|_{L^4_tL^6_x\left(]t_0,0[ \times \mathscr{R}\right)} \leq 4C_8\big\|\baru\big\|_{L^4_tL^6_x\left(]t_0,0[ \times \mathscr{R}\right)}^2. 
\end{align*}

In view of Equation~\eqref{bar-U estimate, epsilon} obtained in the previous step, the smallness condition in Equation~\eqref{smallness condition on bar-U} is satisfied whenever
\begin{equation}\label{epsilon-0}
    \epsilon_0 \leq \frac{\min\left\{C_{8}^{-1},C_{9}^{-1}\right\}}{32 C_6}. 
\end{equation}
In this case, the unique fixed point $V$ satisfies
\begin{equation}\label{V, final}
     \|V\|_{L^4_tL^6_x\left(]t_0,0[ \times \mathscr{R}\right)} \leq 4C_8 (8C_6\e_0)^2 = 256 \,C_6^2C_8\e_0^2 \leq 64C_6\e_0.
\end{equation}

By \cite[Theorem~3.3]{FLR2}, $V$ is the unique strong solution to Equation~\eqref{perturbed stokes} in $L^4_tL^6_x\left(]t_0,0[\times\mathscr{R}\right)$.

\smallskip
\noindent
{\bf Step~7: Weak-strong uniqueness.}  We show that $$U \equiv \baru +V,$$ where $U$ is the given finite energy weak solution to the Navier--Stokes equations, and $V$ is the unique strong solution to Equation~\eqref{perturbed stokes} in  $L^4_tL^6_x\left(]t_0,0[\times\mathscr{R}\right)$ constructed in the previous step. 

Indeed, we set $$Z := V+\baru-U$$ and recall that $U-\baru \in L^4_tL^4_x\left(]t_0,0[\times\mathscr{R}\right)$. Then $Z$ is a very weak solution to the Stokes system
\begin{equation*}
\begin{cases}
\p_t Z - \Delta Z + \na P = 0\quad\text{and}\quad {\rm div}\,Z = F\qquad \text{in } ]t_0,0[\times \mathscr{R},\\
Z = 0\qquad \text{ on } ]t_0,0[\times \p\mathscr{R},\\
Z\big|_{t=t_0} = 0 \qquad \text{ on } \{t_0\} \times \mathscr{R},
\end{cases}
\end{equation*}
with the forcing term
\begin{align*}
    F = U \otimes U - \left(\baru + V\right)\otimes \left(\baru + V\right).
\end{align*}

Observe by H\"{o}lder's inequality that $F \in L^2_tL^2_x\left(]t_0,0[\times \mathscr{R}\right)$. Thanks to Sohr \cite[IV, Theorems~2.3.1 and 2.4.1]{sohr} (see also Albritton--Barker--Prange \cite[Lemma~2.5 and Remark~2.6]{ABP}), $Z$ is in the regularity class $\left(C^0_t L^2_x \cap L^2_tW^{1,2}_{0,x} \right)\left([t_0,0]\times\mathscr{R}\right)$ and satisfies the energy identity:
\begin{align}\label{energy id}
    \frac{1}{2}\int_{\mathscr{R}}|Z(t,x)|^2\,\dd x + \int_{t_0}^t \int_\mathscr{R}|\na Z(t',x)|^2\,\dd x\,\dd t' = \int_{t_0}^t\int_{\mathscr{R}} \left[Z\otimes \left(V+\baru\right)\right] : \na Z\,\dd s\,\dd t'  
\end{align}
for any $t \in [t_0,0[$. See also Proposition~\ref{prop: uniqueness and energy}.

Consider 
\begin{equation*}
    \mathscr{E}(t):= \sup_{t' \in [t_0,t]} \frac{1}{2}\int_{\mathscr{R}}|Z(t',x)|^2\,\dd x + \int_{t_0}^t \int_\mathscr{R}|\na Z(t',x)|^2\,\dd x\,\dd t'.
\end{equation*}
By  Equation~\eqref{energy id} and H\"{o}lder's, Young's and the Sobolev inequalities, one concludes that
\begin{align*}
    \mathscr{E}(t) \leq C_{10}\mathscr{E}(t) \int_{t_0}^t \left\|\left(V+\baru\right)(s,\bullet)\right\|_{L^6_x(\mathscr{R})}^4\,\dd s
\end{align*}
for any $t \in [t_0,0[$. It is easy to see that $\mathscr{E}(t)=0$ for any $t>t_0$ with $t-t_0$ sufficiently small. Thus, a continuity argument yields that $\mathscr{E}(t)\equiv 0$ for all $t \in [t_0,0[$.

\smallskip
\noindent
{\bf Step~8: Iteration for higher regularity via Lady\v{z}enskaja--Prodi--Serrin.} 
Once we obtain $U \in L^4_tL^6_x\left(]t_0,0[\times\mathscr{R}\right)$, we may conclude via a bootstrap argument via the Lady\v{z}enskaja--Prodi--Serrin regularity criterion \cite{Lady, prodi, Ser}, \emph{cf. e.g.}, Barker--Prange \cite[Lemma~12]{BP}, to promote the regularity of $U \in L^4_tL^6_x$ (which is the endpoint of the range of indices entailed by Lady\v{z}enskaja--Prodi--Serrin) to $L^6_tL^6_x$ (which is subcritical), and hence successively to $L^\infty_t L^\infty_x$. See also Seregin--Shilkin--Solonnikov \cite{sss} for a similar bootstrap process for the boundary regularity of suitable weak solutions involving a smallness condition on the pressure.

From the previous steps in this proof, $ U \in L^4_tL^6_x\left(]t_0,0[\times\mathscr{R}\right)$ is the unique very weak solution to the Navier--Stokes equations in $]t_0,0[\times\mathscr{R}$, where $t_0 \in \left]-1,-\frac{7}{8}\right[$ and $\mathscr{R} \subset\Omega$ is a smooth convex subdomain. Moreover, in view of the bounds in Equations~\eqref{bar-U estimate, epsilon}, \eqref{bar-U estimate}, and \eqref{V, final}, as well as the construction of the initial and boundary data $a$, $b$ (see Equations~\eqref{a,b-def} and \eqref{a, b estimate}), we obtain that
\begin{align}
\|U\|_{L^4_tL^6_x\left(]t_0,0[\times\mathscr{R}\right)} \leq C_{11} \left\{ \|a\|_{L^4\left(]t_0,0[ \times \p\mathscr{R}\right)} + \|b\|_{L^4_x\left(\{t_0\} \times \left[\B_{32r_0}(x_\star)\cap\Omega\right]\right)} \right\} \leq C_{12}\e_0,
\end{align}
where $C_{11}$, $C_{12}$ depend only on the $C^2$-geometry of $\Omega$.

It is crucial here that the boundary data $a = U\big|_{]t_0,0[\times\mathscr{R}}$ is  small in the $L^4_tL^4_x$-norm. Fix any $x_\star\in\po$ and let $\mathscr{R}$ be the smooth convex region in Equation~\eqref{R, def} taken with respect to the base point $x_\star$. By  bootstrap arguments in the previous paragraph, we deduce that 
\begin{align*}
\|U\|_{L^6_tL^6_x\left(]t_0,0[\times\mathscr{R}'\right)} \leq M(1+\e) \qquad\text{whenever }0<\e<\e_0.
\end{align*}
Here we denote
\begin{equation*}
    \begin{cases}
    \e:=\|a\|_{L^4\left(]t_0,0[ \times \p\mathscr{R}\right)} + \|b\|_{L^4_x\left(\{t_0\} \times \left[\B_{32r_0}(x_\star)\cap\Omega\right]\right)},\\
    \|U\|_{\left(L^\infty_t L^2_x \cap L^2_t \dot{W}^{1,2}_x\right)\left(]-1,0[\times\Omega\right)} \leq M,\\
    \mathscr{R}' = \mathscr{R} \cap {\bf B}_{\kappa r_0}^+(x_\star)\qquad \text{with a constant $\kappa \in \left]0, 8\right]$  depending only on the geometry of $\mathscr{R}$.}
    \end{cases}
\end{equation*}

Therefore, we may bootstrap again as in \cite{Ser, prodi, Lady, BP, sss} (see Theorem~\ref{thm: LPS}) to deduce that
\begin{align}\label{wp}
\|U\|_{L^\infty_tL^\infty_x(]t_0,0[ \times \mathscr{R}'')} \leq \wp(M,\e)\qquad\text{whenever }0<\e<\e_0,
\end{align}
where $\wp(M,\e)$ is a polynomial in $M$ and $\e$, and 
\begin{align*}
    \mathscr{R}'' = \mathscr{R} \cap {\bf B}_{\frac{\kappa r_0}{2}}^+(x_\star).
\end{align*}
In particular, the region $\overline{\mathscr{R}''}$ intersects the boundary $\po$ tangentially at $x_\star$, and the bound in  Equation~\eqref{wp} is independent of the choice of $x_\star$. By varying $x_\star$ through $\po$ and invoking the interior regularity theorem (see \cite[Theorem~B]{ABP}), we deduce that 
\begin{align*}
\|U\|_{L^\infty_tL^\infty_x(]t_0,0[ \times \Omega)} \leq C_{13}(M,\e_0)\qquad\text{whenever }0<\e<\e_0,
\end{align*}
where $\e_0$ depends only on the $C^2$-geometry of $\Omega$.

It is classically known that essentially bounded weak solutions to the Navier--Stokes equations are $C^\infty$ \cite{Ser}. Hence, the proof of Theorem~\ref{thm: main} is now complete.  \end{proof}

\section{Discussions}\label{sec: concluding remarks}

Several concluding remarks are in order below.

\begin{enumerate}
    \item 
Theorem~\ref{thm: main} remains valid for an unbounded domain $\Omega \subset \R^3$ with compact smooth boundary $\po$.

\item 
The existence and uniqueness for very weak solutions in $L^4_tL^6_x$ with boundary data in $L^4_tL^4_x$ (\emph{i.e.}, Proposition~\ref{prop: very weak} taken from \cite{FKS, FGS}) require the domain $\Omega \subset \R^3$ to be of regularity $C^{2,1}$. For future investigations, we shall explore how to lower the regularity assumptions for $\Omega$. The recent paper \cite{breit} by Breit has proved an epsilon boundary regularity theorem for local suitable weak solutions (see \cite{CKN, Lin}) over low-regularity domains. 

\item 
Our boundary epsilon regularity Theorem~\ref{thm: main} imposes no condition on pressure. This is in stark contrast to the epsilon regularity theorems for the ``local suitable weak solutions'' \emph{\`{a} la} Caffarelli--Kohn--Nirenberg \cite{CKN}. In our case, the control for pressure is implicitly contained in the estimates for the Stokes semigroup in the very weak solutions.

\item
The $L^\infty_tL^\infty_x$-bound for $U$ in Theorem~\ref{thm: main} depends only on $M$ and the $C^2$-geometry of $\Omega$. The dependence on $C^2$-geometry enters the proof via the injectivity radius of the boundary $\po$ in $\overline{\Omega}$ (\emph{i.e.}, the parameter $r_0$ in the proof of Theorem~\ref{thm: main}). Only the higher-regularity bounds for $U$ may depend on the $C^{2+}$-geometry of $\Omega$.

\end{enumerate}



\medskip

\noindent
{\bf Acknowledgement}. The research of Siran Li is supported by NSFC Projects 12331008 and 12411530065, Young Elite Scientists Sponsorship Program by CAST 2023QNRC001, National Key Research $\&$ Development Programs 2023YFA1010900 and 2024YFA1014900, Shanghai Rising-Star
Program 24QA2703600, Qi-Guang Scholarship, and Shanghai Frontiers Science Center of Modern Analysis.

\medskip
\noindent
{\bf Competing Interest Statement}. We declare that there are no conflicts of interest involved.

\medskip
\noindent
{\bf Data Availability Statement}. We declare that no data are associated with this work.


\begin{thebibliography}{99}

\bibitem{ab}
D. Albritton and T. Barker, Localised necessary conditions for singularity formation in the Navier-Stokes equations with curved boundary, \textit{J. Differ. Equ.} \textbf{269} (2020), 7529--7573.

\bibitem{ABP}
D. Albritton, T. Barker, and C. Prange, Epsilon regularity for the Navier--Stokes equations via weak-strong uniqueness, \textit{J. Math. Fluid Mech.} \textbf{25} (2023),  Paper No. 49.

\bibitem{ama1}
H. Amann, Navier–Stokes equations with nonhomogeneous Dirichlet data, \textit{Journal Nonlinear Math. Phys.} \textbf{10} Suppl. 1 (2003), 1--11.


\bibitem{ama2}
H. Amann, \textit{Nonhomogeneous Navier–Stokes equations with integrable low-regularity data}, Int. Math. Ser., Kluwer Academic/Plenum Publishing, New York, 2002, 1--26.

\bibitem{BP}
T. Barker and C. Prange, Localized smoothing for the Navier–Stokes
equations and concentration of critical
norms Near singularities, \textit{Arch. Ration. Mech. Anal.} \textbf{236}, (2020) 1487--1541.


\bibitem{Bog1}
M.~E. Bogovskiĭ, Solution of the first boundary value problem for an equation of continuity of an incompressible medium, {\it Dokl. Akad. Nauk SSSR,} {\bf 248} (1979), no.~5, 1037--1040.

\bibitem{Bog2}
M.~E. Bogovskiĭ, Solutions of some problems of vector analysis, associated with the operators ${\rm div}$\ and ${\rm grad}$, in {\it Theory of cubature formulas and the application of functional analysis to problems of mathematical physics}, pp. 5--40, 149, Proc. Sobolev Sem., {\bf No. 1}, 1980, Akad. Nauk SSSR Sibirsk. Otdel., Inst. Mat., Novosibirsk.

\bibitem{breit}
D. Breit, Partial boundary regularity for the Navier--Stokes equations in irregular domains, \textit{J. Funct. Anal.} \textbf{289} (2025), Article No. 111188.

\bibitem{CKN}
L. Caffarelli, R. Kohn, and L. Nirenberg, Partial regularity of suitable weak solutions of the Navier--Stokes equations, \textit{Comm. Pure Appl. Math.} \textbf{35} (1982), 771--831.

\bibitem{ESS}
L. Escauriaza, G.~A. Seregin, and V. \u{S}ver\'{a}k, $L_{3,\infty}$-solutions of Navier--Stokes equations and backward uniqueness, \textit{Uspekhi Mat. Nauk.} \textbf{58} (2003), 3--44.



\bibitem{FLR1}
E.~B. Fabes, J.~E. Lewis, and N.~M. Rivi\`{e}re, Singular integrals and hydrodynamic potentials, \textit{Amer. J. Math.} \textbf{99} (1977), 601--625.


\bibitem{FLR2}
E.~B. Fabes, J.~E. Lewis, and N.~M. Rivi\`{e}re, Boundary value problems for the Navier--Stokes equations, \textit{Amer. J. Math.} \textbf{99} (1977), 626--668.


\bibitem{FGS}
R. Farwig, G.~P. Galdi, and H. Sohr, A new class of weak solutions of the Navier--Stokes equations with nonhomogeneous data, \textit{J. Math. Fluid Mech.} \textbf{8} (2006), 423--444.

\bibitem{FKS}
R. Farwig, H. Kozono, and H. Sohr, Global weak solutions of the Naveir--Stokes equations with nonhomogeneous boundary data and divergence, \textit{Rend. Sem. Math. Univ. Padova} \textbf{125} (2011), 51--70.

\bibitem{fef}
C.~L. Fefferman, Existence and smoothness of the Navier--Stokes equation, pp. 57–67 in \textit{The millennium prize problems}, edited by J. Carlson \textit{et al.}, Clay Math. Inst., 2006. 

\bibitem{foias}
C. Foias, Une remarque sur l'unicit\'{e} des solutions des \'{e}quations de Navier-Stokes en dimension $n$, \textit{Bull. Soc. Math. France} \textbf{89} (1961), 1--8.

\bibitem{GIP}
I. Gallagher, D. Iftimie, and F. Planchon, Asymptotics and stability for global solutions to the Navier--Stokes equations, \textit{Ann. Inst. Fourier (Grenoble)} \textbf{53} (2003), 1387--1424.


\bibitem{gal1}
G.~P. Galdi, On the energy equality for distributional solutions to Naver--Stokes equations,  \textit{Proc. Amer. Math. Soc.} \textbf{147} (2019), 785--792.

\bibitem{gal2}
G.~P. Galdi, On the relation between very weak and Leray--Hopf solutions to Navier--Stokes equations, \textit{Proc. Amer. Math. Soc.} \textbf{147} (2019), 5349--5359.


\bibitem{giga}
Y. Giga, Solutions for semilinear parabolic equations in $L^p$ and regularity of weak solutions of the Navier--Stokes system, \textit{J. Differ. Equ.} \textbf{62} (1986), 186--212.

\bibitem{GSC}
P. Gyrya and L. Saloff-Coste, Neumann and Dirichlet heat kernels in inner uniform domains, \textit{Astérisque} \textbf{336} (2011), viii+144pp. 

\bibitem{hopf}
E. Hopf, \:{U}ber die Anfangswertaufgabe f\"{u}r die hydrodynamischen Grundgleichungen, \textit{Math. Nachr.} \textbf{4} (1951), 213--231.

\bibitem{Lady}
O.A. Ladyzhenskaya, Uniqueness and smoothness of generalized solutions of Navier-Stokes equations, \textit{Zap. Nau\u{c}. Semin. Leningrad. Otdel. Mat. Inst. Steklov. (LOMI)} \textbf{5} (1967) 169--185.

\bibitem{LS}
O.~A. Ladyzhenskaya and G.~A. Seregin, On partial regularity of suitable weak solutions to the three-dimensional Navier--Stokes equations, \textit{J. Math. Fluid Mech.} \textbf{1} (1999), 356--387.

\bibitem{LR}
Z. Lei and X. Ren, Quantitative partial regularity of the Navier-Stokes equations and applications, \textit{Adv. Math.} \textbf{445} (2024), Paper No. 109654, 40 pp.

\bibitem{leray}
J. Leray, Sur le mouvement d'un liquide visqueux emplissant l'espace, \textit{Acta. Math.} \textbf{63} (1934), 183--248.


\bibitem{Lin}
F. Lin, A new proof of the Caffarelli--Kohn--Nirenberg theorem, \textit{Comm. Pure Appl. Math.} \textbf{51} (1998), 241--257.

\bibitem{lions}
P.-L. Lions, \textit{Mathematical topics in fluid mechanics, I: Incompressible models}, Oxford Lecture Series in Mathematics and its Applications 3, Oxford University Press, New York, 1996. 

\bibitem{NRS}
J. Ne\u{c}as, M. R${\rm \mathring{u}}$\v{z}i\v{c}ka, and V. \u{S}ver\'{a}k, On Leray's self-similar solutions of the Navier--Stokes equations, \textit{Acta Math.} \textbf{176} (1996), 283--294.

\bibitem{prodi}
G. Prodi, Un teorema di unicit\`{a} per le equazioni di Navier--Stokes, \textit{Ann. Mat. Pura Appl.} \textbf{48} (1959), 173--182.

\bibitem{S03}
G.~A. Seregin, Remarks on regularity of weak solutions to the Navier-Stokes equations near the boundary, \textit{Zap. Nauchn. Sem. S.-Peterburg. Otdel. Mat. Inst. Steklov. (POMI)} \textbf{295} (2003), Kraev. Zadachi Mat. Fiz. i Smezh. Vopr. Teor. Funkts. 33, 168–179, 246; translation in
\textit{J. Math. Sci. (N.Y.)} \textbf{127} (2005), 1915--1922.

\bibitem{sss}
G.~A. Seregin, T.~N. Shilkin, and V.~A. Solonnikov, Boundary partial regularity for the Navier-Stokes equations, \textit{Zap. Nauchn. Sem. S.-Peterburg. Otdel. Mat. Inst. Steklov. (POMI)} \textbf{310} (2004), Kraev. Zadachi Mat. Fiz. i Smezh. Vopr. Teor. Funkts. 35 [34], 158–190, 228; translation in
\textit{J. Math. Sci. (N.Y.)} \textbf{132} (2006), 339--358.

\bibitem{Ser}
J. Serrin, On the interior regularity of weak solutions of the Navier--Stokes equations, \textit{Arch. Ration. Mech. Anal.} \textbf{9} (1962), 187--195.

\bibitem{sohr}
H. Sohr, \textit{The Navier--Stokes equations. An elementary functional analytic approach}, Modern Birkh\"{a}user Classics, Birkh\"{a}user/Springer Basel AG, Basel (2001) [2013 reprint of the 2001 original].

\bibitem{Str}
M. Struwe, On partial regularity results for the Navier--Stokes equations, \textit{Comm. Pure Appl. Math.} \textbf{41} (1988), 437--458.

\bibitem{tao}
T. Tao, Localisation and compactness properties of the Navier--Stokes global regularity problem, \textit{Anal. PDE} \textbf{6} (2013), 25--107.

\bibitem{Tsai}
T.~P. Tsai, On Leray's self-similar solutions of the Navier--Stokes equations satisfying local energy estimates, \textit{Arch. Ration. Mech. Anal.} \textbf{143} (1998), 29--51.

\bibitem{Vas}
A.~F. Vasseur, A new proof of partial regularity of solutions to Navier--Stokes equations, \textit{NoDEA Nonlinear Differ. Equ. Appl.} \textbf{14} (2007), 753--785.

\bibitem{vdB}
M. van den Berg, Gaussian bounds for the Dirichlet heat kernel, \textit{J. Funct. Anal.} \textbf{88} (1990), 267--278.

\bibitem{Wang}
L. Wang, Partial regularity for Navier-Stokes equations, \textit{
J. Math. Fluid Mech.} \textbf{27} (2025), Paper No. 26.

\end{thebibliography}
\end{document}